\documentclass[a4paper, oneside, reqno, 12pt]{amsart}
\usepackage[english]{babel}
\usepackage[utf8]{inputenc}
\usepackage{amsmath, amssymb, amsfonts}
\usepackage{mathtools}
\usepackage{fullpage}
\usepackage{comment}
\usepackage{ textcomp, cmap}

\usepackage[usenames,dvipsnames]{xcolor}

\usepackage[backref=page]{hyperref}
\hypersetup{
	colorlinks   = true, 
	urlcolor     = blue, 
	linkcolor    = Blue, 
	citecolor   = Green 
}
\usepackage[capitalise, nameinlink, noabbrev]{cleveref}
\usepackage{autonum}
\usepackage{caption}
\usepackage[font = small]{subcaption}
\usepackage{pgf, tikz}
\usetikzlibrary{arrows}
\usepackage{float} 
\usepackage{amsthm} 
\theoremstyle{plain}
\newtheorem{theorem}{Theorem}

\newtheorem{lemma}[theorem]{Lemma}
\newtheorem{corollary}[theorem]{Corollary}
\theoremstyle{definition}

\newtheorem{definition}{Definition}

\theoremstyle{remark}
\newtheorem{remark}{Remark}

\definecolor{ffqqqq}{rgb}{1.,0.,0.}

\title[Alon--Boppana-type bound]{{Alon--Boppana-type bounds for weighted graphs}}
\author[Alexandr Polyanskii, Rinat Sadykov]{{Alexandr~Polyanskii and Rinat~Sadykov}}

\address{Alexandr Polyanskii,
\newline\hphantom{iii} Institute of Mathematics and Informatics, Bulgarian Academy of Sciences Bulgaria, Sofia 1113, Acad. G. Bonchev Str., Bl. 8
\newline\hphantom{iii} Moscow Institute of Physics and Technology, Institutskiy per. 9, Dolgoprudny, Russia 141700
}
\email{\href{mailto:alexander.polyanskii@gmail.com}{alexander.polyanskii@gmail.com}}
\urladdr{\url{http://polyanskii.com}}

\address{Rinat Sadykov,
\newline\hphantom{iii} Moscow Institute of Physics and Technology, Institutskiy per. 9, Dolgoprudny, Russia 141700
}
\thanks{The research of A.P. was supported by the Ministry of Education and Science of Bulgaria, Scientific 
Programme "Enhancing the Research Capacity in Mathematical Sciences (PIKOM)", 
No. DO1-67/05.05.2022. A.P. is a Young Russian Mathematics award winner and would like to thank its sponsors and jury.}

\begin{document}

\thispagestyle{empty}

\begin{abstract}
The \textit{unraveled ball} of radius $r$ centered at a vertex $v$ in a weighted graph $G$ is the ball of radius $r$ centered at $v$ in the universal cover of $G$. We present a general bound on the maximum spectral radius of unraveled balls of fixed radius in a weighted graph.

The weighted degree of a vertex in a weighted graph is the sum of weights of edges incident to the vertex. A weighted graph is called \textit{regular} if the weighted degrees of its vertices are the same. Using the result on unraveled balls, we prove a variation of the Alon--Boppana theorem for regular weighted graphs.
\end{abstract}

\maketitle

\section{Introduction}
In 1993, Freidmen~\cite{friedman1993geometric} refined the celebrated Alon--Boppana theorem~\cite{AB}. He proved that for every $d$-regular graph $G$ with diameter $2r$, the second largest eigenvalue of adjacency matrix of $G$, denoted by $\lambda_2(G)$, satisfies
\[
\lambda_2(G)\geq 2 \left( 1 - \frac{\pi^2}{r^2}+O(r^4)\right) \sqrt{d-1}.
\]

In 2005, Hoory~\cite[Theorem 1]{hoory2005lower} studied the spectral radius of the universal cover of a non-regular graph. As a corollary he proved a variation of the Alon--Boppana theorem for graphs with $r$-robust average degree at least $d$, which was later improved by Jiang~\cite{jiang2019spectral}; see also~\cite[Lemma~19]{jiang2020forbidden}. We say that a graph has \textit{$r$-robust average degree} at least $d$ if the average degree of a graph is at least $d$ after deleting any ball of radius $r$.

\begin{theorem}{\cite[Theorem 8]{jiang2019spectral}}
\label{theorem jiang}
Let $d \geq 1$ be a real number and let $r$ be a positive integer. If a graph $G$ has an $r$-robust average degree at least $d$, then 
    \[
    \frac{\lambda_2(G)}{\lambda_1(P_r)}\geq \sqrt{d-1}.
    \]
\end{theorem}
We denote by $P_r$ the path with $r$ vertices and spectral radius $\lambda_1(P_r)=2\cos (\frac{\pi}{r+1})$.

To prove Theorem~\ref{theorem jiang}, Jiang studied the maximum spectral radius of unraveled balls of a graph $G$, which are balls in the universal cover of $G$; see Definition~\ref{definition unraveled ball}. In 2022, Wang and Zhang applied the machinery developed in~\cite{jiang2019spectral} and proved an analog of Theorem~\ref{theorem jiang} for the normalized Laplacian of a graph~\cite[Theorem 1.6]{wang2022weighted}, improving the bounds from~\cite{young2022weighted}; see also~\cite{srivastava2018alon}. To show this result, they studied the maximum spectral radius of unraveled balls of a weighted graph.

Motivated by the works of Jiang and Wang--Zhang, we develop further their ideas and prove a generalization of Wang and Zhang's result on the spectral radius of unraveled balls for a weighted graph; see Theorem~\ref{theorem most general version}. This implies an analog of Theorem~\ref{theorem jiang} for regular weighted graphs, as shown in Theorem~\ref{theorem alon boppana for weighted regular graphs}. 

A \textit{weighted graph} is a graph without parallel edges and loops, in which every edge is assigned to a positive number. Formally, a weighted graph $G$ is a triple $(V(G), E(G), w_G)$, where $V(G)$ and $E(G)$ are the vertex and edge sets of the graph $G$, respectively, and $w_G: E(G) \to \mathbb{R}_+$ is the weight function, with $\mathbb R_+$ being the set of positive real numbers. For sake of brevity, we write $w_{ab}$ and $w_{ba}$ for the weight $w_G(ab)$ of an edge $ab\in E(G)$. The \textit{weighted degree} of a vertex $v$, denoted by $w_v$, is the sum of the weights of the edges incident to $v$, that is, 
\(
w_v = \sum_{vu \in E(G)} w_{vu} .
\)
A weighted graph is called $w$-\textit{regular} if the weighted degree of every vertex equals $w$. Throughout the paper, we regularly write ``\textit{a~weighted graph with minimal degree at least 2}'', which means that each vertex is incident to at least 2 edges (rather than the weighted degree of each vertex is at least 2).

A \textit{non-backtracking walk} of length $n$ in a weighted graph is a sequence of vertices $(v_0,\dots, v_n)$ such that any two consecutive are adjacent and $v_i \neq v_{i+2}$ for all $i\in\{0,\dots,n-2\}$. Denote by $W_i(G)$ the set of non-backtracking walks on a graph $G$ of length $i$.

\begin{definition}
\label{definition unraveled ball}
Given a weighted graph $G$, we define the weighted tree $\tilde{G}(v,r)$ as follows. Its vertex set is the set of all non-backtracking walks of length at most $r$ that start at $v$, where two vertices are adjacent if one is a simple extension of the other. Specifically, vertices $(v_0, \dots, v_n)$ and $(u_0,\dots, u_m)$ with $n<m$ are adjacent if and only if $m=n+1$ and $v_i=u_i$ for all $i\in \{0,\dots, n\}$. We say this edge of $\tilde{G}(v,r)$ is extended by the edge $u_{m-1}u_m$ in the graph $G$. Two vertices of the same length are never adjacent.
We assign a weight to each edge in $\tilde{G}(v,r)$ equal to the weight of its extending edge in $G$. 

In other words, the graph $\tilde{G}(v,r)$, which we call an \textit{unraveled ball}, is isomorphic to a ball  of radius $r$ in the universal cover $\tilde{G}$ of $G$. Slightly abusing notation, in the current paper, we say $\tilde{G}(v,r)$ is an induced subgraph of $\tilde{G}$
\end{definition}

It is worth mentioning that we may look at the set $W_1(G)$ as the set of directed edges of a graph $G$, that is, for any edge $xy\in E(G)$, there are two corresponding non-backtracking edges $(x,y)$ and $(y,x)$ in $W_1(G)$. So, we write $w_{(x,y)}$ for $w_{xy}$ if $xy\in E(G)$. 

The weighted adjacency matrix of a weighted graph $G$ is denoted by $A_G$. Let $\lambda_i(G)$ be the $i$-th eigenvalue of $G$. The main corollary of Theorem~\ref{theorem most general version}, the central result of the paper, is presented below.

\begin{theorem}
\label{theorem spectral radius of unraveled balls 2}
    Let $G$ be a $w$-regular weighted graph with minimum degree at least $2$. Then for any positive integer $r$, there is a vertex $v$ such that
\begin{equation}
    \frac{\lambda_1(\tilde{G}(v,r))}{\lambda_1(P_{r+1})} \geq \frac{\sum_{e\in W_1(G)} w_e^{3/2}(w-w_e)^{1/2}}{w|V(G)|}.
\end{equation}
\end{theorem}

We define the average combinatorial degree of a weighted graph $G$ as $\frac{2|E(G)|}{|V(G)|}$. Using Theorem~\ref{corollary strong form for w-regular graphs}, we obtain a variation of the Alon--Boppana theorem.

\begin{theorem}
\label{theorem alon boppana for weighted regular graphs}
Let $G$ be a $w$-regular weighted graph with combinatorial degree $d\geq 39$ and let $r$ be any positive integer. Assume that for every vertex $v$ there is an induced subgraph of $G\setminus G(v,r+1)$ with minimum weighted degree at least ${2w\sqrt{d-1}}/{d}$. Then we have
\[
\frac{\lambda_2(G)}{\lambda_1(P_{r+1})}\geq \frac{w\sqrt{d-1}}{d}.
\]
\end{theorem}

A slightly more involved argument allows to show that Theorem~\ref{theorem alon boppana for weighted regular graphs} holds for $d\geq 7.1980\dots$; see Remark~\ref{central remark}.

The rest of the paper is organized as follows. Section \ref{section main result} presents the proof of the main result of the paper, Theorem~\ref{theorem most general version}, which requires auxiliary notation. In Section \ref{section corollaries}, we derive Corollary~\ref{corollary strong form for w-regular graphs}, slightly generalizing Theorem~\ref{theorem spectral radius of unraveled balls 2}, as well as its slightly weaker form, Corollary~\ref{corollary weak form for w-regular graphs}. Finally, we establish Theorem~\ref{theorem alon boppana for weighted regular graphs} in Section~\ref{section alon-boppana}.

\subsection*{Acknowledgments.} We thank Alexander Golovanov for fruitful discussions and careful proofreading of the paper.

\section{Bounding the maximum spectral radius of an unraveled ball}
\label{section main result}

\begin{definition}
Given a weighted graph $G$ with minimum degree at least 2, a stationary Markov chain $(E_i)_{i=1}^\infty$ on $W_1(G)$ is \textit{assigned to} $G$ if its transition matrix $P=(p_{e_1,e_2})_{e_1,e_2\in W_1(G)}$ satisfies
\begin{equation}
\label{equation transition matrix}
p_{e_1,e_2} = \Pr \big(E_{i}=e_2| E_{i-1}=e_1\big) = 0 \text{ if $e_2$ does not prolong $e_1$}.
\end{equation} 
We say that $e_2\in W_1(G)$ \textit{prolongs} $e_1\in W_1(G)$ if there are three distinct vertices $x,y,z\in V(G)$ such that $e_1=(z,y)$ and $e_2=(y,x)$. For the sake of brevity, we write $e_1\to e_2$ if $e_2$ prolongs $e_1$.
\end{definition}
If the minimum degree of a weighted (connected) graph $G$ is at least 2, then a stationary Markov chain assigned to $G$ has no absorbing states and its stationary distribution $\pi=(\pi_{e})_{e\in W_1(G)}$ is well defined. Since the Markov chain is stationary, we have
\begin{equation}
\label{equation stationary distribution}
\Pr\big(E_i=e\big)=\pi_e \text{ for any positive integer $i$.}
\end{equation}

Since the ending vertex of $E_i$ and the starting vertex of $E_{i+1}$ are the same, we can concatenate $E_1,\dots, E_i$ to form a random non-backtracking walk of length $i$, denoted by the random variables $Y_i=(X_0,\dots, X_i)$.

With these definitions, we can state and prove the following general theorem.

\begin{theorem}
\label{theorem most general version}
Let $G$ be a weighted graph with minimum degree at least 2, let $(E_i)_{i=1}^{+\infty}$ be a stationary Markov chain assigned to $G$ with transition matrix $P=(p_{e_1,e_2})_{e_1,e_2\in W_1(G)}$ and stationary distribution $\pi=(\pi_e)_{e\in W_1(G)}$. For any function $g:W_1(G)\to \mathbb R$ and a positive integer $r$, there is a vertex $v$ of $G$ such that
\[
    \frac{\lambda_1(\tilde{G}(v,r))}{\lambda_1(P_{r+1})}\geq \Big(
\sum\limits_{\substack{e_1,e_2\in W_1(G)\\e_1\to e_2}} w_{e_2}g(e_1)g(e_2) \pi_{e_1}\sqrt{p_{e_1,e_2}}
\Big) 
\Big(
\sum\limits_{e\in W_1(G)} g^2(e)\pi_e
\Big)^{-1}.
\]
\end{theorem}

\begin{proof}

Let us begin by defining the weighted forest $F_G$ as the union of all graphs $\tilde{G}(v,r)$, where $v\in V(G)$, that is,
\[
    F_G=\bigcup_{v\in V(G)}\tilde{G}(v,r).
\] 
Thus, the vertex set of $F_G$ is $\bigcup_{i=0}^{r+1}W_i(G)$. Since $F_G$ is a union of disjoint trees, we have 
\[
    \lambda_1(F_G)=\max\big\{\lambda_1(\tilde{G}(v,r)), v\in V(G)\big\}.
\]
Therefore, it is sufficient to show that 
\[
\frac{\lambda_1(F_G)}{\lambda_1(P_{r+1})} \geq 
\Big(
\sum\limits_{\substack{e_1,e_2\in W_1(G)\\e_1\to e_2}} w_{e_2}g(e_1)g(e_2) \pi_{e_1}\sqrt{p_{e_1,e_2}}
\Big) 
\Big(
\sum\limits_{e\in W_1(G)} g^2(e)\pi_e
\Big)^{-1}.
\]

Denote by $(x_1,\dots, x_{r+1})\in \mathbb R^{r+1}$ the eigenvector of the spectral radius $\lambda_1(P_{r+1})$ of the path $P_{r+1}$ of length $r$. Then the Rayleigh principle yields
\begin{equation}\label{equation Rayleigh for the path of length r}
    \sum\limits_{i = 2}^{r+1} 2x_{i-1}x_i = \lambda_1(P_{r+1}) \sum\limits_{i=1}^{r+1}x_i^2.
\end{equation}

Define a vector $f\in \mathbb R^{V(F)}$ by setting, for $\omega=(v_0,v_1,\dots,v_i)\in W_i(G)$ and $\omega'=(v_{i-1},v_i)\in W_1(G)$,
\[
    f(\omega):= 
    \begin{cases}
    0, & \text{ if } i=0,\\
    x_i g\big(\omega'\big)\sqrt{\Pr(Y_i = \omega)}, & \text{ otherwise.}
    \end{cases}
\]

Therefore, we have
\begin{align}
    \langle f, f\rangle =& \sum\limits_{i = 1}^{r+1} \sum\limits_{\omega \in W_i(G)} f(\omega)^2 \\
    = &\sum\limits_{i = 1}^{r+1} x_i^2 \sum\limits_{\omega \in W_i(G)} g^2(\omega') \Pr(Y_i =\omega)\\ 
    = &\sum\limits_{i = 1}^{r+1} x_i^2 \sum_{\omega'\in W_1(G)} g^2(\omega') \Pr(E_i=\omega')=\\
    =&\sum_{i=1}^{r+1}x_i^2 \sum_{e\in W_1(G)} g^2(e) \pi_e.\label{equation scalar product f,f}
\end{align}
Denoting $\omega^-=(v_0,\dots, v_{i-1})$ and $\omega''=(v_{i-2},v_{i-1})$, we obtain
\begin{align}
\langle f, A_{F_G} f\rangle =& \sum\limits_{i = 2}^{r+1} \sum\limits_{\omega \in W_i(G)} 2x_{i-1}x_iw_{\omega^-} f(\omega)f(\omega^-) \\
    =&\sum\limits_{i = 2}^{r+1} 2x_{i-1}x_{i}\sum\limits_{\omega \in W_i(G)} w_{\omega'}g(\omega'')g(\omega')\sqrt{\Pr(Y_{i-1} =\omega^-)\Pr(Y_i =\omega)}.
\end{align}
For the Markov chain, we have
\begin{equation}\label{equation ratio of probabilities}
    \frac{\Pr(Y_i =\omega)}{\Pr(Y_{i-1} =\omega^-)} = \Pr (E_i=\omega'|E_{i-1}=\omega'') = p_{\omega'',\omega'}.
\end{equation}
Substituting this in the equation for $\langle f, A_{F_G} f\rangle$, we get
\begin{align} \label{equation f,Wf second time}
    \langle f, A_{F_G} f \rangle &= \sum\limits_{i = 2}^{r+1} 2x_{i-1}x_{i}\sum\limits_{\omega \in W_i(G)}
    \frac{w_{\omega'}g(\omega'')g(\omega')}{\sqrt{p_{\omega'',\omega'}}} \Pr(Y_i =\omega)
    \\&=\sum_{i=2}^{r+1} 2x_{i-1}x_i \sum_{\substack{e_1,e_2\in W_1(G)\\e_1\to e_2}} \frac{w_{e_2}g(e_1)g(e_2)}{\sqrt{p_{e_1,e_2}}} \Pr\big(E_i=e_2, E_{i-1}=e_1 \big).
\end{align}
Since
\(
\Pr\big(E_i=e_2, E_{i-1}=e_1 \big)= \Pr\big(E_i=e_2|E_{i-1}=e_1\big) \Pr\big(E_{i-1}=e_1\big)=p_{e_1,e_2}\pi_{e_1}
\),  
we easily conclude that 
\[
    \langle f, A_{F_G} f \rangle = \sum_{i=2}^{r+1} 2x_{i-1}x_i \sum_{\substack{e_1,e_2\in W_1(G)\\e_1\to e_2}} w_{e_2}g(e_1)g(e_2) \pi_{e_1}\sqrt{p_{e_1,e_2}}.    \label{equation scalar product f,Wf}
\]
Combining this equality with \eqref{equation Rayleigh for the path of length r}, \eqref{equation scalar product f,f},  and the Rayleigh principle $\lambda_1(F_G)\geq \langle f, A_{F_G} f \rangle/\langle f,f \rangle$, we finish the proof.
\end{proof}

\begin{remark}
\label{remark universal cover general bound}
Recall that $\tilde{G}(v,r)$ is a ball of radius $r$ centered at $v$ in the universal cover $\tilde{G}$ of $G$. So by the monotonicity of spectral radius, $\lambda_1(\tilde{G})\geq \lambda_1(\tilde{G}(v,r))$. Since $\lambda_1(P_{r+1})\to 2$ as $r\to +\infty$, under the assumptions of Theorem~\ref{theorem most general version}, we obtain
\[
    \lambda_1(\tilde{G}) \geq 2\Big(
\sum\limits_{\substack{e_1,e_2\in W_1(G)\\e_1\to e_2}} w_{e_2}g(e_1)g(e_2) \pi_{e_1}\sqrt{p_{e_1,e_2}}
\Big) 
\Big(
\sum\limits_{e\in W_1(G)} g^2(e)\pi_e
\Big)^{-1}.
\]
\end{remark}

\section{Proofs of corollaries on regular weighted graphs}
\label{section corollaries}

To prove special cases of Theorem~\ref{theorem most general version}, the authors of~\cite[Thereom~1]{jiang2019spectral} and~\cite[Theorem~1.5]{wang2022weighted} consider the following stationary Markov chain on $W_1(G)$ such that its stationary distribution can be easily found. Namely, they assumed that given the stage $E_i=(v_{i-1},v_i)$, the stage $E_{i+1}$ is chosen among $\{(v_i,u)\in W_1(G):u\ne v_{i+2}\}$ uniformly at random. Hence the transition matrix $P=(p_{e_1,e_2})_{e_1,e_2\in W_1(G)}$ of this Markov chain is defined by
\[
p_{(x,y),(z,t)}=
\begin{cases}
\dfrac{1}{\deg y - 1} &\text{if $y=z$ and $t\ne x$};\\
0 &\text{ otherwise.}
\end{cases}
\]
One can easily verify that the distribution \(\pi=(\pi_e)_{e\in W_1(G)}\) with $\pi_e=\frac{1}{|W_1(G)|}$ is stationary. 

In the next general corollary of Theorem~\ref{theorem most general version}, we use another stationary Markov chain assigned to a regular weighted graph such that its stationary distribution be explicitly found.

\begin{corollary}
\label{corollary strong form for w-regular graphs}

Let $G$ be a $w$-regular weighted graph with minimum degree at least $2$. For any function $g:W_1(G)\to \mathbb R$ and any positive integer $r$, there is a vertex $v$ of $G$ such that
\begin{equation}
    \frac{\lambda_1(\tilde{G}(v,r))}{\lambda_1(P_{r+1})} \geq \Big(\sum_{\substack{e_1,e_2\in W_1(G)\\e_1\to e_2}} g(e_1)g(e_2)w_{e_1}w_{e_2}^{3/2}(w-w_{e_1})^{1/2}\Big)\Big(\sum_{e\in W_1(G)} g^2(e) w_{e}(w-w_e)\Big)^{-1}.
    \label{equation general bound}
\end{equation}
Particularly, choosing $g(e)=(w-w_e)^{-1/2}$, we have
\begin{equation}
    \frac{\lambda_1(\tilde{G}(v,r))}{\lambda_1(P_{r+1})} \geq \frac{\sum_{e\in W_1(G)} w_e^{3/2}(w-w_e)^{1/2}}{w|V(G)|}.
\end{equation}
\end{corollary}
\begin{proof}
Consider a stationary Markov chain assigned to $G$ defined by its transition matrix $P=(p_{e_1,e_2})_{e_1,e_2\in W_1(G)}$ as follows
\[
    p_{e_1,e_2}=
    \begin{cases}
    \dfrac{w_{e_2}}{w-w_{e_1}} & \text{if $e_1\to e_2$};\\
    0 & \text{otherwise.}
    \end{cases}
\]
Using $\pi = \pi P$ for the stationary distribution $\pi=(\pi_e)_{e\in W_1(G)}$, one can easily verify that
\[
    \pi_e=\frac{w_e(w-w_e)}{S}, \text{\ \ where\ } S= \sum_{e\in W_1(G)} w_e(w-w_e).
\]
Applying Theorem~\ref{theorem most general version} to this Markov chain, we easily obtain the first desired inequality.

Assuming that $g(e)=(w-w_e)^{-1/2}$, we have
\begin{equation}
\frac{\lambda_1(\tilde{G}(v,r))}{\lambda_1(P_{r+1})}\geq \Big( \sum_{\substack{e_1,e_2\in W_1(G)\\e_1\to e_2}} w_{e_1}w_{e_2}^{3/2}(w-w_{e_2})^{-1/2}\Big)\Big(\sum_{e\in W_1(G)}w_e\Big)^{-1}.
\label{equation trivial bound}
\end{equation}

The multipliers of the right-hand side of \eqref{equation trivial bound} can be easily found
\begin{align}
    \sum_{e\in W_1(G)} w_e=&w|V(G)|.\\
    \sum_{\substack{e_1,e_2\in W_1(G)\\e_1\to e_2}}w_{e_1}^{1/2}w_{e_2}(w-w_{e_1})^{-1/2}=&\sum_{e_1\in W_1(G)} w_{e_1}^{3/2}(w-w_{e_1})^{-1/2}\sum_{e_2:e_1\to e_2}w_{e_2}
    \\=&\sum_{e\in W_1(G)} w_e^{3/2}(w-w_e)^{1/2}.  
\end{align}
Substituting these equalities in \eqref{equation trivial bound}, we obtain the second desired inequality.
\end{proof}

\begin{corollary}
\label{corollary weak form for w-regular graphs}
    Let $G$ be a $w$-regular weighted graph with minimum degree at least $2$. Assume that its average combinatorial degree $d\geq 2$ satisfies $2\sqrt{d-1}/d\leq \mu$, where $\mu = \frac{3-\sqrt{3}}{4}$ (that is, $d \geq 38.7620\dots$). Then, for any positive integer $r$, we have
    \[
    \frac{\lambda_1(\tilde{G}(v,r))}{\lambda_1(P_{r+1})}\geq \frac{w\sqrt{d-1}}{d}.
    \]
\end{corollary}
\begin{proof}
There are two possible cases. 

\textit{Case 1.} There is $e=(v,u)\in W_1(G)$ such that $w_e \geq \mu w$. Define a vector $f'\in \mathbb R^{V(\tilde{G}(v,r))}$ by setting
\[
f'(x)=
\begin{cases}1&\text{ if $x\in\{ (v),(v,u)\}$;}\\0&\text{ otherwise.}\end{cases}
\]
Using the Rayleigh principle, we have 
\[
\lambda_1(\tilde{G}(v,r))\geq 
\frac{\langle f', A_{\tilde{G}(v,r)} f'\rangle}
{\langle f', f'\rangle}=w_{e}\geq \mu w,
\]
which finishes the proof in this case as $2\sqrt{d-1}/d\leq \mu$ and $\lambda_1(P_{r+1})<2$.
\medskip

\textit{Case 2.} For any vertex $e\in W_1(G)$, we have $w_e\leq \mu w$. By a straightforward computation, we obtain that the function $x\mapsto x^{3/2}(w-x)^{1/2}$ is convex on the interval $[0,\mu w]$ and concave on $[\mu w, w]$, where $\mu=\frac{3-\sqrt{3}}{4}$. By Jensen's inequality, we obtain
\[
    \sum_{e\in W_1(G)} w_e^{3/2}(w-w_e)^{1/2} \geq  |W_1(G)|\Big( \frac{w}{d}\Big)^{3/2}\Big(w-\frac{w}{d}\Big)^{1/2} = |V(G)|\frac{w^2\sqrt{d-1}}{d}.
\]
Substituting this inequality in the second inequality of Corollary \ref{corollary strong form for w-regular graphs}, we finish the proof.
\end{proof}

\begin{remark}
\label{central remark}
Slightly modifying the argument, we can prove Corollary~\ref{corollary weak form for w-regular graphs} for $d\geq 1/t_0=7.1980\dots$, where $t_0$ is defined below. 

Let $g:[0,w]\to \mathbb R$ be a function given by $g(y)=y^{3/2}(w-y)^{1/2}$. For any $t\in (0,1)$, let $\ell_{t}:\mathbb R\to \mathbb R$ be an affine function defining the tangent line of $g$ at the point $tw\in (0,w)$; see Figure~\ref{fig:h(x)}. Consider the following system of equations and inequalities
\[
\begin{cases}
\ell_t(xw)=g(xw);\\
x=\frac{2\sqrt{1/t-1}}{1/t};\\
0<t<\mu=\frac{3-\sqrt{3}}{4}<x,
\end{cases}
\]
which has only one solution: $x=x_0=0.6917\dots, t=t_0=0.1389\dots$; see Figure~\ref{fig:h(x)}.
\begin{figure}
    \centering
    \begin{tikzpicture}[x=10cm, y=15cm]
    \draw[->] (-0.045, 0) -- (1.1, 0) node[right] {$x$};
    \draw[->] (0, -0.03) -- (0, 0.4) node[above] {$y$};
    \draw[domain=0:1, samples = 5000,  variable = \x, blue] plot ({\x}, {\x * sqrt(\x) * sqrt(1-\x)} );
    \draw[blue] (1, 0) -- (0.9999, 0.01);
    \draw [domain=0:0.8, samples = 2, variable = \x, red] plot ({\x}, {0.49087 * \x - 0.0201444} );
    \draw[dashed]  (0.1389, 0.4) -- (0.1389, 0) node [below]{$t_0w$};
    \draw[dashed]  (0.6917, 0.4) -- (0.6917, 0) node [below]{$x_0w$};
    \draw[dashed]  (0.316987, 0.4) -- (0.316987, 0) node [below]{$\mu w$};
\end{tikzpicture}
    \caption{Graphs of $g$ (blue) and $\ell_t$ (red).}
    \label{fig:h(x)}
\end{figure}
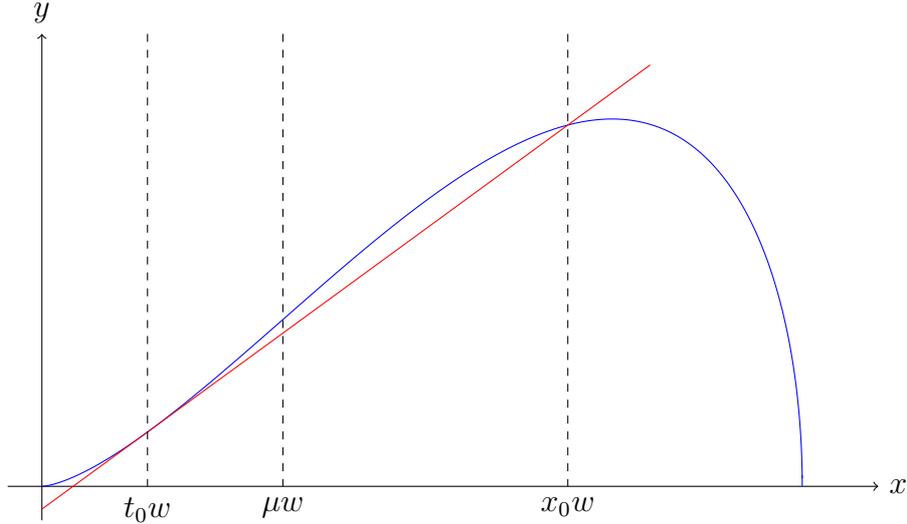

First, we may assume that
\[
    w_e\leq wx_0=\frac{2w\sqrt{1/t_0-1}}{1/t_0}
\] 
for all $e\in W_1(G)$; otherwise, we use the proof of the first case in Corollary~\ref{corollary weak form for w-regular graphs}. 

Next, consider the function $h:[0, x_0w]\to \mathbb R$ defined by
\[
h(x)=
\begin{cases}
g(x) &\text{ if }0 \leq x \leq t_0w;\\
\ell_{t_0}(x) &\text{ if } t_0w< x\leq x_0w.
\end{cases}
\]
Recall that the function $g$ is convex of $[0,\mu w]$. Since $0<t_0<\mu=\frac{3-\sqrt{3}}{4}$ and $\ell_{t_0}$ defines the tangent line for the graph of $g$ at the point $t_0w$, we conclude that the function $h$ is convex on $[0,x_0w]$ and $g(y)\geq h(y)$ for any $0\leq y\leq x_0w$. Therefore, we can apply Jensen's inequality for $h$ and obtain the desired inequality
\begin{align}
    \sum_{e\in W_1(G)} w_e^{3/2}(w-w_e)^{1/2} =& \sum_{e\in W_1(G)} g(w_e) \\\geq& \sum_{e\in W_1(G)} h(w_e)
    \\\geq&  |W_1(G)|h\big( \frac{w}{d} \big) \\=& |W_1(G)| g\big(\frac{w}{d}\big) = |V(G)|\frac{w^2\sqrt{d-1}}{d}.
\end{align}
(Here we use that $d\geq 1/t_0$, and thus $h(\frac{w}{d})=g(\frac{w}{d})$ by the definition of $h$.) Substituting this inequality in the second inequality of Corollary \ref{corollary strong form for w-regular graphs}, we finish the proof.
\end{remark}
\begin{remark}
    Using Corollaries~\ref{corollary strong form for w-regular graphs} and \ref{corollary weak form for w-regular graphs}, we can prove lower bounds for the spectral radius for the universal cover $\tilde{G}$ of a $w$-regular weighted graph $G$ as it is shown in Remark~\ref{remark universal cover general bound}.
\end{remark}
\section{Proof of Theorem~\ref{theorem alon boppana for weighted regular graphs}}
\label{section alon-boppana}
The following lemma connects the spectral radii of balls $G(v,r)$ and $\tilde{G}(v,r)$.
\begin{lemma}{\cite[Lemma 4.2]{wang2022weighted}}
    \label{lemma connection between H and tilde H}
    For any vertex $v$ of a graph $H$ and any positive integer~$r$, we have
    \[
    \lambda_1(H(v,r))\geq \lambda_1(\tilde{H}(v,r)).
    \]
\end{lemma}
Suppose that $G$ has a vertex of degree 1, that is, incident to one edge. Then there is a connected component of $G$ that is a path of length 1 with weight of its only edge equal to $w$. Clearly, this component has eigenvalue $w$, which is large enough to finish the proof in this case. Therefore, we can assume that there are no vertices of degree 1 in $G$. 

Lemma~\ref{lemma connection between H and tilde H} and Corollary~\ref{corollary weak form for w-regular graphs} yield that there exists a vertex $v\in G$ such that
\[
\lambda_1(G(v,r))\geq \lambda_1(\tilde{G}(v,r)) \geq  \lambda_1(P_{r+1})\frac{w\sqrt{d-1}}{d}.
\]
Denote by $f_1\in \mathbb R^{V(G)}$ the vector that coincides on $V(G(v,r))$ with the eigenvector of the spectral radius of $G(v,r)$ and is zero on $V(G)\setminus V(G(v,r))$. By the Rayleigh principle, we have 
\[
\lambda_1(G(v,r))=\frac{\langle f_1, A_G f_1\rangle}{\langle f_1,f_1\rangle}.
\]

Let $G'$ be an induced subgraph of $G\setminus G(v,r+1)$ with minimum weighted degree at least $2w\sqrt{d-1}/{d}$. Define a vector $f_2\in \mathbb R^{V(G)}$  by setting
\[
f_2(x)=\begin{cases}
1&\text{ if $x\in V(G')$};\\
0&\text{ otherwise}.
\end{cases}
\]
Hence by the Rayleigh principle, we obtain
\[
\lambda_1(G')\geq \frac{\langle f_2, A_{G} f_2\rangle}{\langle f_2,f_2\rangle} \geq 2w \frac{\sqrt{d-1}}{d}.
\]

One can choose scalars $c_1$ and $c_2$ such that the vector $f=c_1 f_1 +c_2 f_2\ne 0$ is perpendicular to the eigenvector $(1,\dots, 1)$ of the spectral radius  $\lambda_1(G)=w$. Therefore, by the Rayleigh principle, we obtain
\[
\lambda_2(G) \geq \frac{\langle f, A_Gf \rangle}{\langle f, f \rangle} = \frac{c_1^2 \langle f_1, A_G f_1 \rangle + c_2^2\langle f_2, A_Gf_2 \rangle}{c_1^2\langle f_1, f_1 \rangle + c_2^2\langle f_2, f_2 \rangle} \geq\lambda_1(P_{r+1})\frac{w\sqrt{d-1}}{d},
\]
which finishes the proof.\hfill $\square$

\bibliographystyle{abbrv}
\bibliography{bib.bib}
\end{document}